\documentclass[preprint,12pt]{elsarticle}
\usepackage{tikz}
\usepackage{amsthm,amsfonts,amssymb,amscd,amsmath,enumerate,verbatim,calc,graphicx,geometry}
\usepackage[all]{xy}
\newtheorem{theorem}{Theorem}[section]
\newtheorem{lemma}[theorem]{Lemma}
\newtheorem{proposition}[theorem]{Proposition}
\newtheorem{corollary}[theorem]{Corollary}
\theoremstyle{definition}
\theoremstyle{definitions}
\newtheorem{definition}[theorem]{Definition}

\newtheorem{remark}[theorem]{Remark}
\newtheorem{example}[theorem]{Example}

\theoremstyle{notations}

\theoremstyle{remarks}

\journal{}

\begin{document}

\begin{frontmatter}



\title{Topologized standard construction and locally quasinormal subgroups }


\author[]{Z. Pashaei \corref{cor1}}
\ead{szpashaei1986@yahoo.com}
\author[]{N. Gorentas}
\ead{ngortas@yahoo.com }
\author[]{R. Abdi}
\ead{a.roghayeh86@gmail.com}

\address{Department of Mathematics, Faculty of Sciences, Van Yuzuncu Yil University, Van,\\
Turkey, 65080-Campus,\\}
\cortext[cor1]{Corresponding author}
\begin{abstract}
This paper is the extended version of some results in \cite{SPasha, SSPasha}. Let $H\leq \pi_{1}(X,x_{0})$. The first part of the paper is devoted to studying weaker conditions under which homotopically Hausdorff relative to H becomes homotopically path Hausdorff relative to H. By using of these conditions, we explore the connection between whisker and quotient topologies on fundamental group. After that, we address the coincidence of two determined topologies on the standard construction $\widetilde{X}_{H}$ when H is a locally quasinormal subgroup. Finally, Example \ref{EX} illustrates that these kinds of subgroups are more extensive than normal subgroups and justifies the generalization of these results.
\end{abstract}


\end{frontmatter}

% \linenumbers

%\\\\\\\\\\\\\\\\\\\\\\\\\\\\\\\\\\\\\\\\\\\\\\\\\\\\\\\\\\\\\\\\\\\\\\\\\\\\\\\\\\\\\\\\\\\\\\\\\\\\\\\\\\\\\\\\\\\\\\\\\\\\\\\\\\\\\\\\\
%=========================================================================================================================================
%/////////////////////////////////////////////////////////////////////////////////////////////////////////////////////////////////////////

\section{Introduction}

In the classical covering theory, semilocally simply connectivity is a crucial condition. Indeed, when $X$ is Peano semilocally simply connected, connected covering spaces of $X$ correspond with the conjugacy classes of all subgroups of $\pi_{1}(X,x)$. Accordingly, it is possible there are many subgroups of $\pi_{1}(X,x)$  which are not correspondence to a covering map but which are correspondence to other generalizations of covering map \cite{Brazss, Braz}. There is a natural topology on fundamental group,  $\pi_{1}^{qtop}(X,x)$, which plays important role in the existence of covering spaces: $X$ admits a universal covering space if and only if $\pi_{1}^{qtop}(X,x)$ is discrete. So, the entire subgroups of  $\pi_{1}(X,x)$ are covering subgroup if $\pi_{1}^{qtop}(X,x)$ is discrete .

If $X$ is a non-semilocally simply connected space, such as 1-dimensional Menger universal cure, the Hawaiian Earring and other complicated local spaces, we do not have simply connected covering. Accordingly, one is led to generalize the concept of universal covering space including such spaces. Considering those properties of covering spaces which are essential is a joint approach. One of these generalizations is semicoverings. Semicoverings are in connection with topological group structures on fundamental groups \cite{BrazT, Brazss}. Next approach, named generalized covering map, are expressed only on the basis of unique lifting properties and need not to be a semicovering map \cite{ZastrowC}. Besides what is said, the topological properties of covering, semicovering, and generalized covering subgroups of $\pi_{1}^{qtop}(X,x)$ have been studied in \cite{Brazss, BrazO, TorS}. The existence of universal connected covering space of $X$ makes the coincidence of these three concepts occur. In \cite{BrodT}, Brodskiy et al. have studied whisker topology on fundamental group for the first time, $\pi_{1}^{wh}(X,x)$, and they have shown that $\pi_{1}^{wh}(X,x)$ does not depend on the choice of point $x\in{X}$ in case $X$ is small loop transfer (SLT for short) space. A few results of the paper \cite{BrodT} are relevant to strong SLT spaces which are stronger version of SLT spaces. The authors in \cite{SPasha, SSPasha} have introduced spaces are more extensive than SLT and strong SLT spaces. One of advantages of these new approaches is in relation to the vastness of them and their generalizations. Example 2.16 of \cite{SSPasha} shows that (strong) SLT spaces are wider than semilocally simply connected and small loop spaces. Moreover, (strong) SLT spaces and their generalizations have a number of other advantages over semilocally simply connected and small loop approaches. Let us recall some of these basic results which have been recently obtained by researchers.

\begin{itemize}
 \item A Peano topological space $X$ is SLT iff $\widetilde{X}^{wh}_{e}=\widetilde{X}^{top}_{e}$. The extended version of it is Theorem 3.2 of \cite{SPasha}; moreover, we can verify that $\pi_{1}^{qtop}(X,x_{0})$ is topological group whenever $X$ is SLT at $x_{0}$.
\item A path connected space $X$ is strong SLT iff $\widetilde{X}^{wh}_{e}=\widetilde{X}^{l}_{e}$. The extended version of it is Theorem 4.2 of \cite{SSPasha}.
\item An equivalent condition for the discreteness of $\pi_{1}^{wh}(X,x_{0})$ is that $X$ be semilocally simply connected at $x_{0}$.
\item  Letting $X$ be an SLT space, the concepts of h.H and h.p.H are equivalent; its relative version does also hold. 
\item Each generalized covering subgroup of $\pi_{1}(X,x_{0})$, e.g. H, is semicovering subgroup when $X$ is an H-SLT space at $x_{0}$.
\item Each semicovering subgroup of $\pi_{1}(X,x_{0})$, e.g. H, is covering subgroup when $X$ is a strong H-SLT space at $x_{0}$.
\end{itemize}

Note that the property of locally quasinormal subgroup, defined in \cite{FG} (see Definition \ref{De1}), led us to improve some results of the articles \cite{SPasha, SSPasha}. At the begining of section 3, our purpose is unifying two important concepts as mentioned above. By using of this coincidence, we get information about the connection between whisker and quotient topologies on fundamental group. However, we use new conditions to expand Theorem 4.2 of \cite{SSPasha}. Finally, Example \ref{EX} illustrates that locally quasinormal subgroups are more extensive than normal subgroups which are one of the requirements of some theorems of the articles \cite{SPasha, SSPasha}.

%In section 3, by using of the topological properties of the standard construction $\widetilde{X}_{H}$ we show when $H\leq \pi_{1}(X,x_{0})$ is a covering, semicovering, and generalized covering subgroup. In addition, some seperation axioms  $\widetilde{X}_{H}$ equipped with the Whisker and Lasso topologies are investigated.

\section{Definitions and terminologies}

Throughout the paper $(X,x_{0})$ will denote a pointed path-connected space and H will denote a subgroup of fundamental group $\pi_{1}(X,x_{0})$. However, we call $X$ is Peano when it is connected and locally path connected. For a given path $\alpha:[0,1]\rightarrow X$, $\bar{\alpha}(t)=\alpha(1-t)$ is the reverse path. Let $\gamma$ and $\delta$ be paths in $X$. The concatenation of $\gamma$  and $\delta$ is denoted by $\gamma \ast \delta$ in which $\gamma(1)=\delta(0)$. Moreover, we denote constant path sending the unit interval set $[0,1]$ to $x$ by $c_{x}$. For given $x\in{X}$, $P(X,x)$ denotes the subspace of paths whose starting point is $x$ and $\Omega(X,x)$ denotes the subspace of paths whose the initial and final points are equal to $x$. Letting $\alpha\in{P(X,x_{0})}$, we denote path-conjugate subgroup $[\bar{\alpha}H \alpha]=\lbrace [\bar{\alpha}\ast \delta \ast \alpha] \ \vert \ [\delta]\in{H} \rbrace$ of $ \pi_{1}(X,x)$ by $H_{\alpha}$. The map $f_{\#}: \pi_{1}(X,x_{0}) \rightarrow \pi_{1}(Y,y_{0})$ denotes the homomorphism induced by continuous function $f:(X,x_{0})\rightarrow (Y,y_{0})$. The subgroup $\pi(\mathcal{U},x_{0})$ of $\pi_{1}(X,x_{0})$, named Spanier subgroup, is generated by elements having the forms as $\alpha \ast \beta \ast \bar{\alpha}$, where $\beta\in{\Omega(X,\alpha(1))}$ in which Im($\beta$) is contained in some elements of $\mathcal{U}$. In the following theorem, Spanier has shown that Spanier subgroups help us to determine when a map is covering. Note that H is called a covering subgroup if $X$ has a covering map such as $p:\widetilde{X}\rightarrow X$ with $p(\tilde{x}_{0})=x_{0}$ so that $p_{\#}\pi_{1}(\widetilde{X}, \tilde{x}_{0})$ and $H$ are equal.
%where $f_{\#}([\alpha])=[f\circ \alpha]$.

%Note that $H$ is called a covering subgroup if there is a covering map $p:\widetilde{X}\rightarrow X$ with $p(\tilde{x}_{0})=x_{0}$ such that $p_{\#}\pi_{1}(\widetilde{X}, \tilde{x}_{0})=H$.

\begin{theorem}\cite[Theorem 2.5.13]{SPan}\label{Th}
Given a Peano space $X$, $H\leq \pi_{1}(X,x_{0})$ is a covering subgroup iff there exists an open cover $\mathcal{U}$ of $X$ such that $\pi(\mathcal{U},x_{0})\leq H$.
\end{theorem}

The standard construction is introduced by Spanier in \cite{SPan} when he was going to classify covering spaces of Peano space $X$ having at least one univesal covering space. Take $\alpha, \beta\in{P(X,x_{0})}$. We say that $\alpha$ and $\beta$ have the same equivalence classes, denoted by $\alpha\sim \beta$,  if and only if $\beta(1)=\alpha(1)$ and $[\alpha]\in{H[\beta]}$. Define $\widetilde{X}_{H}=P(X,x_{0})/\sim$. Let $[\alpha]_{H}$ denote the equivalence class of $\alpha$. Put $\tilde{x}_{H}=[c_{x_{0}}]_{H}$. We write $\widetilde{X}$ instead of $\widetilde{X}_{H}$ whenever H is trivial; it is called the standard construction. Let us recall that three types of topology have been studied $ \widetilde{X}_{H} $ so far. The quotient map $q:P(X,x_{0})\rightarrow \widetilde{X}_{H}$ induces the quotient topology on $\widetilde{X}_{H}$, denoted by $\widetilde{X}_{H}^{top}$, in which $P(X,x_{0})$ is equipped with the compact-open topology. In the attempt to construct covering spaces, topology Spanier defined on $\widetilde{X}_{H}$ is as follows; it is named the whisker topology by some people \cite{BrodC, VZcom} and denoted by $\widetilde{X}_{H}^{wh}$.
% where $P(X,x_{0})$ equips with compact-open topology.
\begin{definition}
The Whisker topology on standard construction has the basis
\begin{center}

$N_{H}([\alpha]_{H}, U)=\lbrace [\gamma]_{H} \ \vert \  \gamma \simeq \alpha \ast \delta \ for \ some \ \delta \ inside \  open \ subset \ U \ of \ \alpha(1)=\delta(0) \rbrace$

\end{center}
%where $ [\alpha]_{H}\in{\widetilde{X}_{H}} $ and $U$ is an open neighborhood of $\alpha(1)$. 
\end{definition}

\begin{definition}
The Lasso topology on standard construction has the basis
\begin{center}

$N_{H}([\alpha]_{H}, \mathcal{U}, U)=\lbrace [\beta]_{H} \ \vert \  \beta \simeq \alpha \ast \gamma \ast \delta \ for \ some \ [\gamma]\in{\pi(\mathcal{U},\alpha(1))} \ and  \ for \ some \ \delta \ inside \ U \ of \ \alpha(1)=\delta(0)\rbrace$

\end{center}
where $U\in{\mathcal{U}}$. It is denoted as $\widetilde{X}^{l}_{H}$.
\end{definition}

It can be easily observed that $\pi_{1}(X,x_{0})$ is a subset of  $\widetilde{X}$. The induced topologies on $\pi_{1}(X,x_{0})$ by $\widetilde{X}^{wh}$,  $\widetilde{X}^{top}$, and  $\widetilde{X}^{l}$ are denoted as $\pi_{1}^{wh}(X,x_{0})$, $\pi_{1}^{qtop}(X,x_{0})$, and $\pi_{1}^{l}(X,x_{0})$, respectively (see \cite{BrazT, BrazO, BrodT, SPasha, SSPasha, VZcom} for more details). 

Semicovering maps are defined based on the local homeomorphism property (see \cite{BrazS, Brazss}). As in the definition of covering subgroup, H is a semicovering subgroup if it can be expressed in terms of a semicovering map such as $p:\widetilde{X}\rightarrow X$. It was shown that semicoverings correspond to open subgroups of quasitopological fundamental group $\pi_{1}^{qtop}(X)$ \cite{BrazS}. The authors in \cite{TorS} have tried to recognize which one of subgroups are open. This attempt led them to define a special subgroups are rather similar to Spanier subgroups. At first, they introduced path open cover $\mathcal{V}= \lbrace V_{\alpha} \ \vert \ \alpha\in{P(X,x_{0})} \ and \ \alpha(1)\in{V_{\alpha}} \rbrace$ which is an open cover of $X$. The subgroup $\widetilde{\pi}(\mathcal{V}, x_{0})\leq \pi_{1}(X,x_{0})$, called path Spanier subgroup, is generated by elements having the forms as $\alpha \ast \beta \ast \bar{\alpha}$, where $\beta$ is a loop at $\alpha(1)$ whose image is contained in $V_{\alpha}\in{\mathcal{V}}$.

\begin{theorem}\cite[Corollary 3.3]{TorS}\label{Th2}
For a given Peano space $X$, H is a semicovering subgroup if and only if it contains a path Spanier subgroup.
\end{theorem}

Unlike universal covering maps, generalized universal covering maps, initially defined by Fischer and Zastrow in \cite{Zastrow}, play  important role in finding fundamental group of complicated local spaces, e.g. Hawaiian Earring ( see \cite{Zastrow, FZ} for more details). After, this definition has been extended to general subgroups of $\pi_{1}(X,x_{0})$ by Brazas in \cite{Braz} as follows. 

\begin{definition}
A map $p:(\widetilde{X}, \tilde{x})\rightarrow (X,x)$ is called a generalized covering map if it has the following properties:
\begin{itemize}
\item[1.] $\widetilde{X}$ is a Peano space,
\item[2.] for every map $g:(Z,z)\rightarrow (X,x)$, there exists a unique map $\tilde{g}:(Z,z)\rightarrow (\widetilde{X}, \tilde{x})$ such that $p\circ \tilde{g}=g$ provided that $g_{\#}\pi_{1}(Z,z)\subseteq p_{\#}\pi_{1}(\widetilde{X}, \tilde{x})$
\end{itemize}
%Let $\widetilde{X}$ be a Peano space. A map $p:(\widetilde{X}, \tilde{x}_{0})\rightarrow (X,x_{0})$ is called to be generalized covering map if for every pointed Peano space $(Y,y_{0})$ and map $f:(Y,y_{0})\rightarrow (X,x_{0})$ such that $f_{\#}\pi_{1}(Y,y_{0})\subseteq p_{\#}\pi_{1}(\widetilde{X}, \tilde{x}_{0})$, there exists a unique map $\tilde{f}:(Y,y_{0})\rightarrow (\widetilde{X}, \tilde{x}_{0})$ so that $p\circ \tilde{f}=f$.
\end{definition}

In this extension, the unique lifting property of covering maps has been just used. Note that these two notions are not necessarily equal (see \cite[Proposition 3.6]{Zastrow}, \cite[Corollary 2.5.14]{SPan}). In \cite[Lemma 5.10]{Braz}, it is verified that each generalized covering map such as $p:(\widetilde{X}, \tilde{x}_{0})\rightarrow (X,x_{0})$ associated to a subgroup H, it means that $p_{\#}\pi_{1}(\widetilde{X}, \tilde{x}_{0})=H$,  is equivalent to a specific map, named endpoint projection map, $p_{H}: (\widetilde{X}_{H}^{wh}, \tilde{x}_{H}) \rightarrow (X,x_{0})$ with $p_{H}([\alpha]_{H})=\alpha(1)$. In other words, the topology of any generalized covering space coincides with the standard topology. 
%Note that the unique path lifting property of $p_{H}$ is sufficient to be a generalized covering map.

\begin{definition}\cite[p. 190]{Zastrow} \label{Dh}
Let $\alpha\in{P(X,x_{0})}$ with $\alpha(1)=x$ and $g\notin{H}$. If there is an open subset $U\subseteq X$ containing $x$ such that $i_{\#}\pi_{1}(U,x)_{\bar{\alpha}}\cap Hg=\emptyset$, we say that $X$ is  homotopically Hausdorff (h.H for short) relative to H.
%The space $X$ is called  homotopically Hausdorff (h.H for short) relative to H in case for every $x\in{X}$, for every $\alpha\in{P(X,x_{0})}$ with $\alpha(1)=x$, and for each $g\notin{H}$, there is an open subset $U$ at $x$ so that $i_{\#}\pi_{1}(U,x)_{\bar{\alpha}}\cap Hg=\emptyset$.

%there is not any closed path $ \gamma:I\rightarrow U $ at $x$ with $ [ \alpha\ast\gamma\ast\bar{\alpha}]\in Hg$.
\end{definition}

\begin{definition} \cite{BrazO}\label{Dhp}
Let $ \alpha, \beta \in P(X,x_0) $  with $ \alpha(1)=\beta(1)=x $  and $ [\alpha] \notin H[\beta] $. Also, suppose that we have partition $\lbrace [t_{i-1},t_{i}] \vert \ 1\leq i\leq n \ such \ that \ t_{0}=0, \  t_{n}=1\rbrace$ of the unit interval $I$ and open subsets $ U_{1}, U_{2}, ..., U_{n} $ with $\alpha([t_{i-1},t_{i}])\subseteq U_{i} $ for $ i=1,2,...,n $. If $ \lambda\in P(X,x_0)$ is another path with $\lambda(1)=x$ in which $ \lambda([t_{i-1},t_{i}])\subseteq U_{i}$ and $\lambda(t_{i})=\alpha(t_{i})$ so that for $i=0,1,..., n $, $ [\lambda] \notin H[\beta]$, then we call $X$ is homotopically path Hausdorff (h.p.H for short) relative to H.

 %The space $X$ is called homotopically path Hausdorff (h.p.H for short) relative to H, if for $ \alpha, \beta \in P(X,x_0) $  with $ \alpha(1)=\beta(1) $  and $ [\alpha] \notin H[\beta] $, there are partition $\lbrace [t_{i-1},t_{i}] \vert \ i=1,2,...,n \ such \ that \ t_{0}=0, \  t_{n}=1\rbrace$ of  $[0,1]$ and open subsets $ V_{1}, V_{2}, ..., V_{n}$ 
% there is a partition $ 0=t_{0}<t_{1}<...<t_{n}=1 $ and a sequence of open subsets $ V_{1}, V_{2}, ..., V_{n} $ with  $ \alpha([t_{i-1},t_{i}])\subseteq V_{i} $ such that if $ \lambda:I \rightarrow X $  is another path satisfying $  \lambda([t_{i-1},t_{i}])\subseteq U_{i} $ for $i=1,2,...,n$ and $ \lambda(t_{i})=\alpha(t_{i})  $ for every $i=0,1,..., n $, $ [\lambda] \notin H[\beta]$.
\end{definition}

Note that one of the requirements of the closeness of H in $\pi_{1}^{qtop}(X,x_{0})$ is that $p_{H}$ has the unique path lifting property.

\begin{theorem}\label{Th28} \cite{BrazO}
Let $X$ be a Peano space. If H is closed in $\pi_{1}^{qtop}(X,x_{0})$, then $p_{H}$ has the unique path lifting property.
\end{theorem}

\begin{theorem}\label{Th29}
If $H$ is closed in $\pi_{1}^{qtop}(X,x_{0})$, then $X$ is h.p.H relative to H. If $X$ is Peano and h.p.H relative to H, then H is closed in $\pi_{1}^{qtop}(X,x_{0})$.
\end{theorem}

The propostion below refers to necessary and sufficient conditions for $p_{H}$ becomes a generalized covering map.

\begin{proposition}\cite{Braz}\label{Pr28}
Let $H\leq\pi_{1}(X,x_{0})$. Then
\begin{itemize}
\item If $p_{H}$ is generalized covering, $X$ is h.H relative to H.

\item $p_{H}$  is generalized covering if $X$ is h.p.H relative to H.
\end{itemize}
\end{proposition}

In \cite{BrodT}, Brodskiy et al. initially inroduced the concept of (strong) small loop transfer spaces. The extended version of them is defined based on an arbitrary subgroup of $\pi_{1}(X,x_{0})$. 
%The authors have called these new versions  H-small loop transfer ($H$-SLT for short) (see \cite[Definition 2.11]{SPasha}) and strong $H$-small loop transfer ( strong $H$-SLT for short) (see \cite[Definition 1.3]{SSPasha}).

\begin{definition}\cite[Definition 2.11]{SPasha}
Let $ \alpha \in{P(X,x_{0})}$ with $\alpha(1)=x$. If for each open subset $U$ at $x_{0}$ there exists an open subset $V$ at $x$ such that $i_{\#}\pi_{1}(V, \alpha(1))_{\bar{\alpha}}\subseteq Hi_{\#}\pi_{1}(U,x_{0})$, we call $X$ is an H-small loop transfer (H-SLT for short) space at $x_{0}$. However, we say that $X$ is an H-SLT space if for each $\delta\in{P(X,x_{0})}$ with $\delta(1)=x$, $X$ is $H_{\delta}$-SLT at $x$. We write SLT instead of H-SLT whenever H is trivial.
\end{definition}

\begin{definition}\cite[Definition 1.3]{SSPasha}
Let for each $x\in X$ and for every open subset $U$ at $x_{0}$ there exists an open subset $V$ at $x$ so that for every $\alpha\in{P(X,x_{0}})$ with $\alpha(1)=x$ we have $i_{\#}\pi_{1}(V, \alpha(1))_{\bar{\alpha}}\subseteq Hi_{\#}\pi_{1}(U,x_{0})$. Then we call $X$ is a strong H-SLT space at $x_{0}$. However, we say that $X$ is a strong H-SLT space if for each  $\delta\in{P(X,x_{0})}$ with $\delta(1)=x$, $X$ is strong $H_{\delta}$-SLT at $x$. As in the above definition, we write strong SLT instead of strong H-SLT whenever H is trivial.
\end{definition}
 %%%%%%%%%%%%%%%%%%%%%%%%%%%%%%%%%%%%%%%%%%%%
%===========================================================
%/////////////////////////////////////////////////////////////////////////////////////////////////////////////////////
%----------------------------------------------------------------------------------------------------------------------

\section{Main results} 

Definitions \ref{Dh} and \ref{Dhp} and their relations have been investigated by some people in \cite{Braz, FG, Zastrow, SPasha}. One of significant features of these kinds of spaces can be seen in Proposition \ref{Pr28}. Though every h.p.H relative to $H$ is h.H relative to $H$, but the converse need not to be true (see \cite{FZ, VZA}). It is of importance to determine when these two notions are coincident. Recall that the property of being semilocally simply connected causes the coincidence of these concepts. Even the authors in \cite[Theorem 2.5]{SPasha} showed that this statement holds for small loop transfer spaces relative to H provided that H is normal. One of main purposes of this article is expanding this theorem. In fact, we consider another subgroup instead of normal subgroup; it is called locally quasinormal subgroup.

\begin{definition}\label{De1}\cite{FG}
Let  $\alpha\in{P(X,x_{0})}$ with $\alpha(1)=x$. If for each open subset $x\in{U}$ there exists an open subset $V\subseteq U$ of $x$ such that $H\pi(\alpha,V)=\pi(\alpha,V)H$, we say that H is locally quasinormal.
\end{definition}

\begin{lemma}\label{L1}
Let $\alpha\in{P(X,x_{0})}$. If H is a locally quasinormal subgroup, then so is $H_{\alpha}$.
\end{lemma}
\begin{proof}
Let $\delta$ be an arbitrary path from $x$ to $y$. Because H is locally quasinormal, we have $H\pi(\alpha \ast \delta, V)=\pi(\alpha \ast \delta, V)H$. Now, we show that $H_{\alpha}\pi(\delta, V)=\pi(\delta, V)H_{\alpha}$. At first, take an element $[\bar{\alpha}\ast h \ast \alpha \ast \delta \ast \beta \ast \bar{\delta}]\in{H_{\alpha}\pi(\delta, V)}$, where $[h]\in{H}$ and $[\beta]\in{\pi_{1}(V,y)}$. Note that $[\bar{\alpha}\ast h \ast \alpha \ast \delta \ast \beta \ast \bar{\delta}]=[\bar{\alpha}\ast h \ast \alpha \ast \delta \ast \beta \ast \bar{\delta} \ast \bar{\alpha}\ast \alpha]$. Clearly, $[h \ast \alpha \ast \delta \ast \beta \ast \bar{\delta} \ast \bar{\alpha}]\in{H\pi(\alpha \ast \delta, V)}$. According to the above relation, there exist $[h']\in{H}$ and $[\beta']\in{\pi_{1}(V,y)}$ such that $[h \ast \alpha \ast \delta \ast \beta \ast \bar{\delta} \ast \bar{\alpha}]=[\alpha \ast \delta \ast \beta' \ast \bar{\delta}\ast \bar{\alpha}\ast h']\in{\pi(\alpha \ast \delta, V)H}$. Therefore, $[\bar{\alpha}\ast h \ast \alpha \ast \delta \ast \beta \ast \bar{\delta}]=[\bar{\alpha}\ast h \ast \alpha \ast \delta \ast \beta \ast \bar{\delta}\ast \bar{\alpha}\ast \alpha]=[\bar{\alpha}\ast \alpha \ast \delta \ast \beta' \ast \bar{\delta}\ast \bar{\alpha}\ast h' \ast \alpha]=[\delta \ast \beta' \ast \bar{\delta}\ast \bar{\alpha}\ast h' \ast \alpha]\in{\pi(\delta, V)H_{\alpha}}$ which implies that $H_{\alpha}\pi(\delta, V)\subseteq \pi(\delta, V)H_{\alpha}$. In similar way, we can follow $\pi(\delta, V)H_{\alpha}\subseteq H_{\alpha}\pi(\delta, V)$.
\end{proof}

\begin{remark}\label{Re}
Assume that H is a locally quasinormal subgroup. %As we know, for every path $\alpha\in{P(X,x_{0})}$ and for every open neighborhood $U$ at $\alpha(1)=x$  there is an open subset $V\subseteq U$ such that $H\pi(\alpha, V)=\pi(\alpha,V)H$.
 By the proof of lemma \ref{L1}, if we put constant path $c_{x}$ instead of $\delta$, then $H_{\alpha}\pi(c_{x}, V)=\pi(c_{x},V)H_{\alpha}$. Note that $\pi(c_{x}, V)=i_{\#}(V,x)$, where $i:V\rightarrow X$ is a inclusion map.
\end{remark}
 
\begin{theorem}\label{Th44}
Let $X$ be an $H$-SLT space at $x_{0}$ and $H$ be locally quasinormal. If $X$ is h.H relative to $H$, then $X$ is h.p.H relative to $H$.
\end{theorem}

\begin{proof}
Suppose that $[\beta \ast \bar{\alpha}]\notin{H}$ where $\alpha, \beta \in{P(X,x_{0})}$ and $\alpha(1)=\beta(1)=x$. Since $X$ is h.H relative to H, we have an open subset $U\subseteq X$ of $x_{0}$ such that $i_{\#}\pi_{1}(U,x_{0})\cap H[\beta \ast \bar{\alpha}]=\emptyset$. On the other hand, because $H$ is locally quasinormal subgroup, there is an open subset $V\subseteq U$ of $x_{0}$ such that $Hi_{\#}\pi_{1}(V,x_{0})=i_{\#}\pi_{1}(V,x_{0}) H$ and $i_{\#}\pi_{1}(V,x_{0})\cap H[\beta \ast \bar{\alpha}]=\emptyset$. For every $t\in{[0,1]}$, consider the path $\alpha_{t}=\alpha\mid_{[0,t]}$ from $x_{0}$ to $\alpha_{t}(1)=\alpha(t)=x_{t}$ and also put $\alpha_{t_{0}}=c_{x_{0}}$. By assumption, we have open subset $V_{t}\subseteq X$ at $x_{t}$ such that for any closed path $\beta_{t}$ at $x_{t}$ in $V_{t}$ there is a closed path $\gamma_{t}$ at $x_{0}$ in $V$ such that $[\alpha_{t} \ast \beta_{t} \ast \bar{\alpha}_{t}]_{H}=[\delta_{t}]_{H}$ or equivalently, $[\alpha_{t} \ast \beta_{t} \ast \bar{\alpha}_{t}] \in{H[\delta_{t}]}$. By the compactness of closed interval $I=[0,1]$ and the continuity of $\alpha$, we have a partition $\lbrace [t_{i-1},t_{i}] \vert \ i=1,2,...,n \ such \ that \ t_{0}=0, \  t_{n}=1\rbrace$ of  $[0,1]$ and open subsets $U_{1}, ..., U_{n}$ such that $\alpha  [t_{i-1},t_{i}]  \subseteq U_{i}$. Put $\alpha_{i}:=\alpha|_{[t_{i-1}, t_{i}]}$ for $i=1,2,...,n$. Choose another path such as $\gamma$ from $x_{0}$ to $x$ such that image of $\gamma_{i}:=\gamma|_{[t_{i-1}, t_{i}]}$  is contained in $U_{i}$ for $i=1,2,...,n$ and $\gamma(t_{i})=\alpha(t_{i})$ for $i=0,1,...,n$. If we put $\theta_{i}=\alpha_{t_{i-1}} \ast \gamma_{i} \ast \bar{\alpha}_{i}\ast \bar{\alpha}_{t_{i-1}}$ for $1\leq i\leq n$. It is not difficult to see that $\theta_{i}$'s are loop at $\alpha(t_{i-1})$ in $U_{i}$. As stated, for $\theta_{1}, ...., \theta_{n}$ we have, respectively, $[\delta_{1}],..., [\delta_{n}]$ belong to $i_{\ast}\pi_{1}(V, x_{0})$ such that $[\theta_{i}]\in{H[\delta_{i}]}$ for $i=1,2,...,n$. Since $Hi_{\#}\pi_{1}(V,x_{0})=i_{\#}\pi_{1}(V,x_{0}) H$, $H[\delta_{1}]H[\delta_{2}]\in{HHi_{\#}\pi_{1}(V,x_{0})}[\delta_{2}]=Hi_{\#}\pi_{1}(V,x_{0})$ and in a silmilar way  $H[\delta_{1}]H[\delta_{2}]H[\delta_{3}]\in{Hi_{\#}\pi_{1}(V,x_{0})}H[\delta_{3}]=HHi_{\#}\pi_{1}(V,x_{0})[\delta_{3}]=Hi_{\ast}\pi_{1}(V,x_{0})$. By continuity this process, $H[\delta_{1}]H[\delta_{2}]...H[\delta_{n}]\in{Hi_{\#}\pi_{1}(V,x_{0})}$. Note that $\gamma \ast \bar{\alpha}=\theta_{1}\ast...\ast \theta_{n}$. Therefore, $[\gamma \ast \bar{\alpha}]=[\theta_{1}\ast...\ast \theta_{n}]\in{H[\delta_{1}]H[\delta_{2}]...H[\delta_{n}]}=Hi_{\#}\pi_{1}(V,x_{0})$. In other words, there exists a closed path $\delta$ at $x_{0}$ in $V$ such that $[\gamma \ast \bar{\alpha}]\in{H[\delta]}$, i.e., $[\gamma \ast \bar{\alpha}\ast \bar{\delta}]\in{H}$. We have $[\gamma \ast \bar{\beta}]=[\gamma \ast \bar{\alpha}\ast \bar{\delta}][\delta\ast\alpha\ast \bar{\beta}]$. Since $i_{\#}\pi_{1}(V,x_{0})\cap H[\beta \ast \bar{\alpha}]=\emptyset$, so $[\delta\ast\alpha\ast \bar{\beta}]\notin{H}$. Therefore, $[\gamma \ast \bar{\beta}]\notin{H}$ because $[\gamma \ast \bar{\alpha}\ast \bar{\delta}]\in{H}$. This means that $X$ is h.p.H relative to $H$. 
\end{proof}

\begin{corollary}\label{Cor45}
Assume that $H$ is locally quasinormal and $X$ is H-SLT. If $X$ is h.H relative to $H_{\alpha}$, then $X$ is h.p.H relative to $H_{\alpha}$ for every $\alpha\in{P(X,x_{0})}$.
\end{corollary}
\begin{proof}
By Remark \ref{L1}, $H_{\alpha}$ is locally quasinormal. However, by assumption, $X$ is $H_{\alpha}$-SLT at $x$. Accordingly, Theorem \ref{Th44} concludes that $X$ is h.p.H relative to $H_{\alpha}$.
\end{proof}

\begin{remark}
In view of Corollary \ref{Cor45}, the requirements ``H-SLT" and ``locally quasinormality of H" assure that for every $\alpha\in{P(X,x_{0})}$ the concepts of h.p.H relative to $H_{\alpha}$ and h.H relative to $H_{\alpha}$ are coincident. In case of H=1, we also have the coincidence of h.H and h.p.H when $X$ is a SLT space.
\end{remark}

\begin{theorem}\label{Th48}
Let $X$ be an H-SLT space. Then $X$ is h.H relative to $H_{\alpha}$ iff $H_{\alpha}$ is closed in $\pi_{1}^{wh}(X,\alpha(1))$ for every $\alpha\in{P(X,x_{0})}$.
\end{theorem}

\begin{proof}
``Only If'': Take $\alpha\in{P(X,x_{0})}$ with $\alpha(1)=x$. It is shown in \cite[Proposition 3.9]{Ab} that the property of being h.H relative to $H_{\alpha}$ implies that $H_{\alpha}$ is closed in $\pi_{1}^{wh}(X,x)$. 

``If'': By assumption, $X$ is $H_{\alpha}$-SLT at $x$. From Theorem 2.6 of \cite{SPasha}, if $H_{\alpha}$ is closed in $\pi_{1}^{wh}(X,x)$, then $X$ is h.H relative to $H_{\alpha}$. 
\end{proof}

The normality of H is used in some results of \cite{SPasha}, e.g. Corollaries 2.8 and 2.9. These results can be improved as follows.

\begin{corollary}
Suppose $X$ is a locally path connected H-SLT space and H is locally quasinormal. Then $H_{\alpha}$ is closed in $\pi_{1}^{qtop}(X,x)$ iff $H_{\alpha}$ is closed in $\pi_{1}^{wh}(X,x)$.
\end{corollary}

\begin{proof}
Only the sufficiency requires proof. From Theorem \ref{Th48}, $X$ is h.H relative to $H_{\alpha}$ and Corollary \ref{Cor45} implies that $X$ is h.p.H relative to $H_{\alpha}$. Therefore, from Theorem \ref{Th29}, $H_{\alpha}$ is closed in $\pi_{1}^{qtop}(X,x)$.
\end{proof}

The usefulness of the endpoint projection maps can be seen in Lemma 5.10 of \cite{Braz}. However, recall that the unique lifting property of $p_{H}$ results from its unique path lifting property \cite[Lemma 5.9]{Braz}. Indeed, a map $p:\widetilde{X}\rightarrow X$ with $p(\tilde{x}_{0})=x_{0}$ and $H=p_{\#}\pi_{1}(\widetilde{X}, \tilde{x}_{0})$ is a generalized covering map if $p_{H}$ has the unique path lifting property. In \cite{BrazO}, it has been discovered that specified subgroups of fundamental group make the endpoint projection map becomes unique path lifting, e.g. closed subgroups of $\pi_{1}^{qtop}(X,x_{0})$. The following corollary demonstrates that Corollary 2.9 of \cite{SPasha} holds for locally quasinormal subgroups.

\begin{corollary}
Suppose that  $\alpha\in{P(X,x_{0})}$, $X$ is a locally path connected H-SLT, and H is locally quasinormal subgroup. Then $H_{\alpha}$ is closed in $\pi_{1}^{qtop}(X,\alpha(1))$ iff $p_{H_{\alpha}}$ has the unique path lifting property.
\end{corollary}

\begin{proof}
``Only if'': It is immediate from Theorem \ref{Th28}.
%Let $\alpha\in{P(X,x_{0})}$ with $\alpha(1)=x$. We know that if $X$ is locally path connected and $H_{\alpha}$ is closed in $\pi_{1}^{qtop}(X,x)$, then $p_{H_{\alpha}}$ has the unique path lifting property (see Theorem \ref{Th28}). 

``If'': By Proposition \ref{Pr28}, $X$ is h.H relative to $H_{\alpha}$. Since $X$ is H-SLT, Theorem \ref{Th48} concludes that $X$ is h.p.H relative to $H_{\alpha}$. Therefore, Theorem \ref{Th29} implies that $H_{\alpha}$ is closed in $\pi_{1}^{qtop}(X,x)$.
\end{proof}

The connection between whisker and lasso topologies on homotopy class of paths, $\widetilde{X}$, has been introduced by Virk and Zastrow for the first time \cite{VZcom}. After, in similar fashion, the authors in \cite{SSPasha} not only clarified  the connection between $\widetilde{X}^{wh}_{H}$ and $\widetilde{X}^{l}_{H}$ but also characterized conditions for which they become coincident \cite[Theorem 4.2]{SSPasha}; they are not necessarily identical \cite{VZcom}. One of these conditions is the normality of H. We show that this coincidence holds for locally quasinormal subgroups. 

\begin{theorem}\label{Th1}
Let H be locally quasinormal. Then $X$ is a strong H-SLT space iff $\widetilde{X}^{l}_{H_{\alpha}}=\widetilde{X}^{wh}_{H_{\alpha}}$ for each path $\alpha\in{P(X,x_{0})}$.
\end{theorem}

\begin{proof}
``Only if'': From the definitions of whisker and lasso topologies, it is evident that $\widetilde{X}^{l}_{N}$ is coarser than $\widetilde{X}^{wh}_{N}$ for each subgroup $N$ of fundamental group. At first, it will be shown that  $\widetilde{X}^{wh}_{H}$ is coarser than $\widetilde{X}^{l}_{H}$. To do this, we take an open subset $ ([\alpha]_{H},U)$ of $\widetilde{X}^{wh}_{H}$. The hypothesis of locally quasinormality of H assures the existence of an open subset $V\subseteq U$ of $\alpha(1)=x$ such that $H\pi(\alpha, V)=\pi(\alpha,V)H$. Clearly, We have $ ([\alpha]_{H},V)\subseteq([\alpha]_{H},U)$. So, since $X$ is strong H-SLT, for every point $y\in{X}$ there is an open subset $W$ at $y$ such that for every path $\delta$ from $x$ to $y$, for every closed path $\beta:I\rightarrow W$ at $y$ there is a closed path $\lambda:I\rightarrow V$ at $x$ such that $[\delta \ast \beta \ast \bar{\delta}]_{H_{\alpha}}=[\lambda]_{H_{\alpha}}$. Assume that $\mathcal{W}$ is an open cover of $X$ consisting of $W$'s. Define open basis neighborhood $([\alpha]_{H}, \mathcal{W}, V)$ in $\widetilde{X}^{l}_{H}$. Consider $[\alpha\ast l \ast \lambda]_{H}\in{([\alpha]_{H}, \mathcal{W}, V)}$, where $[l]\in{\pi(\mathcal{W}, \alpha(1))}$ and $\lambda:I\rightarrow V$ is a path with $\lambda(0)=\alpha(1)=x$. We know that $l=\Pi_{i=1}^{n} \alpha_{i}\ast \beta_{i}\ast \bar{\alpha}_{i}$, where $\alpha_{i}$'s are paths with $\alpha_{i}(1)=x$ and $\beta_{i}$'s are loops at $\alpha_{i}(1)$ in some $W\in{\mathcal{W}}$. Put $\theta_{i}=\alpha_{i}\ast \beta_{i}\ast \bar{\alpha}_{i}$ for $i=1,2,...,n$. Since $X$ is strong H-SLT, $[\theta_{i}]\in{H_{\alpha}i_{\#}\pi_{1}(V,x)}$ for $i=1,2,...,n$. So, we have $[l]=[\theta_{1}\ast \theta_{2} \ast...\ast \theta_{n}]\in{(H_{\alpha}i_{\#}\pi_{1}(V,x))(H_{\alpha}i_{\#}\pi_{1}(V,x))...(H_{\alpha}i_{\#}\pi_{1}(V,x))}$. By Remark \ref{Re}, $(H_{\alpha}i_{\#}\pi_{1}(V,x))(H_{\alpha}i_{\#}\pi_{1}(V,x))...(H_{\alpha}i_{\#}\pi_{1}(V,x))=H_{\alpha}i_{\#}\pi_{1}(V,x)$ and therefore $[l]\in{H_{\alpha}i_{\#}\pi_{1}(V,x)}$. In other words, there exists a loop $\theta$ in $V$ at $x$ so that $[l]\in{H_{\alpha}[\theta]}$, i.e., $[l \ast \bar{\theta}]\in{H_{\alpha}}$. Write $[l\ast \bar{\theta}]=[l\ast \lambda \ast \bar{\lambda} \ast \bar{\theta}]$. So we have $[l\ast \lambda \ast \bar{\lambda} \ast \bar{\theta}]\in{[\bar{\alpha}H\alpha]}$ or equivalently, $[\alpha \ast l \ast \lambda \ast \bar{\lambda} \ast \bar{\theta} \ast \bar{\alpha}]\in{H}$. This means that $[\alpha\ast l \ast \lambda]_{H}=[\alpha \ast \theta \ast \lambda]_{H}\in{([\alpha]_{H},V)} \subseteq([\alpha]_{H},U)$. Therefore, $([\alpha]_{H}, \mathcal{W}, V) \subseteq([\alpha]_{H},U)$ shows that $\widetilde{X}^{l}_{H}$  is finer than $\widetilde{X}^{wh}_{H}$. So, $\widetilde{X}^{l}_{H}=\widetilde{X}^{wh}_{H}$. Note that we can easily derive that for every path $\alpha\in{P(X,x_{0})}$ the space $X$ is a strong $H_{\alpha}$-SLT space. Accordingly, Lemma  \ref{L1} and the above statements follow that $\widetilde{X}^{wh}_{H_{\alpha}}=\widetilde{X}^{l}_{H_{\alpha}}$ for every path $\alpha\in{P(X,x_{0})}$.

``If'': The proof is analogous to the proof of \cite[Theorem 4.2]{SSPasha}.

%Since $X$ is strong H-SLT, there exists loop $\lambda_{i}$ for $1\leq i\leq n$ in $V$ such that $[\alpha_{i}\ast \beta_{i}\ast \bar{\alpha_{i}}]_{H_{\alpha}}=[\lambda_{i}]_{H_{\alpha}}$, that is, $[\alpha_{i}\ast \beta_{i}\ast \bar{\alpha_{i}}]\in{H_{\alpha}[\lambda_{i}]}$ for $i=1,2,...,n$. $[l]=[\Pi_{i=1}^{n} \alpha_{i}\ast \beta_{i}\ast \bar{\alpha_{i}}]\in{(H_{\alpha}[\lambda_{1}])(H_{\alpha}[\lambda_{2}])...(H_{\alpha}[\lambda_{n}])}$, i.e., $[l]=[k_{1}\ast \lambda_{1}\ast k_{2} \ast \lambda_{2}\ast...\ast k_{n} \ast \lambda_{n}]$, where $[k_{i}]\in{H_{\alpha}}$ for $1\leq i \leq n$. By Remark \ref{Re}, there is loop $\lambda_{1}'$ in $V$ at $x$ such that $[\lambda_{1}\ast k_{2}]=[k_{2}' \ast \lambda_{1}']$, where $[k_{2}']\in{H_{\alpha}}$. Thus, $[l]=[k_{1}\ast k_{2}' \ast \lambda_{1}' \ast \lambda_{2} \ast ... \ast  k_{n} \ast \lambda_{n}]$. Similarly, there exists a loop $\lambda_{2}'$ in $V$ at $x$  such that $[\lambda_{1}' \ast \lambda_{2} \ast k_{3}]=[k_{3}' \ast \lambda_{2}']$. So we have $[l]=[k_{1}\ast k_{2}' \ast k_{3}' \ast \lambda_{2}' \ast \lambda_{3}\ast...\ast  k_{n} \ast \lambda_{n}]$. Hence, after using the same way for other parts, there is a closed path $\theta$ in $V$ at $x$ such that $[l]=H_{\alpha}[\theta]$, i.e., $[l\ast \bar{\theta}]\in{H_{\alpha}}$

\end{proof}

The corollary below is an extended version of Corollary 4.3 of \cite{SSPasha}. 
\begin{corollary}
Suppose H is locally quasinormal and $\alpha \in{P(X,x_{0})}$ with $\alpha(1)=x$. If $X$ is strong H-SLT, then $ (p^{-1}_{H_{\alpha}}(x))^{wh}=(p^{-1}_{H_{\alpha}}(x))^{l} $.
\end{corollary}
\begin{proof}
It follows immediately from Theorem \ref{Th1}.
\end{proof}

%\begin{proof}
%Assume that $X$ is strong H-SLT at $x_{0}$. With a little care in the proof of Theorem \ref{Th1}, we can see that the notion of the strong H-SLT at $x_{0}$ has been applied in the proof of $\widetilde{X}^{l}_{H}=\widetilde{X}^{wh}_{H}$. This relation concludes that $ (p^{-1}_H(x_0))^{wh}=(p^{-1}_H(x_0))^{l} $ because $p^{-1}_H(x_0)\subseteq \widetilde{X}_{H}$. The converse can be implied from Theorem \ref{Th1}.
%\end{proof}

The intersection of all Spanier subgroups of $\pi_{1}(X,x_{0})$, denoted by $\pi_{1}^{sp}(X,x_{0})$, is called Spanier group \cite[Definition 2.3]{FZ}. However, the set of all homotopy classes of small loops of $\pi_{1}(X,x_{0})$ forms a subgroup  which is denoted by $\pi_{1}^{s}(X,x_{0})$ \cite[Definition 1]{V}; note that a loop $\alpha\in{\Omega(X,x_{0})}$ is called small iff it has a homotopy representative in each open subset $U$ of $x_{0}$. The usefulness of these subgroups and other important subgroups can be observed in \cite{FZ, V}. In the following. it is investigated the equality of these two subgroups. Of course, recall that we have the relation $\pi_{1}^{s}(X,x_{0})\leq\pi_{1}^{sp}(X,x_{0})$.

\begin{proposition}\label{Pro1}
Let $\pi_{1}^{s}(X,x_{0})$ contains locally quasinormal subgroup H. Then $\pi_{1}^{s}(X,x_{0})=\pi_{1}^{sp}(X,x_{0})$ if $X$ is strong H-SLT at $x_{0}$.
\end{proposition}

\begin{proof}
Assume that $[\theta]\in{\pi_{1}^{sp}(X,x_{0})}$ and $U$ is an open subset in $X$ containing $x_{0}$. By Remark \ref{Re}, we have an open subset $x_{0}\in{V}\subseteq U$ such that $i_{\#}\pi_{1}(V,x_{0})H=Hi_{\#}\pi_{1}(V,x_{0})$. Since $X$ is strong H-SLT at $x_{0}$, we can define an open cover $\mathcal{U}$ of $X$ such that $\pi_{1}^{sp}(X,x_{0})\subseteq \pi(\mathcal{U},x_{0})$ and $[\theta]\in{\pi(\mathcal{U},x_{0})}$. We know that $\theta=\Pi _{i=1}^{n} \gamma_{i}$, where $\gamma_{i}=\alpha_{i}\ast \beta_{i}\ast \bar{\alpha}_{i}$ for $i=1,2,...,n$, $\alpha_{i}$'s are paths from $x_{0}$ to $\alpha_{i}(1)$, and $\beta_{i}$'s are closed paths in some $U_{i}\in{\mathcal{U}}$ at $\alpha_{i}(1)$. Hence, there is a closed path $\lambda_{i}$ in $V$ at $x_{0}$ such that $[\gamma_{i}\ast \bar{\lambda_{i}}]\in{H}$, that is, $[\gamma_{i}]\in{H[\lambda_{i}]}$. Thus, $[\theta]=[\gamma_{1} \ast \gamma_{2} \ast ...\ast \gamma_{n}]\in{(H[\lambda_{1}])(H[\lambda_{2}])...(H[\lambda_{n}])}$. By the relation $i_{\#}\pi_{1}(V,x_{0})H=Hi_{\#}\pi_{1}(V,x_{0})$, there is a closed path $\gamma:I\rightarrow V$ at $x_{0}$ such that $[\theta]\in{H[\gamma]}$, i.e., $[\theta]=[h\ast \gamma]=[h][\gamma]$. Since $H\subseteq \pi_{1}^{s}(X,x_{0})$, so $[h]\in{i_{\#}\pi_{1}(V,x_{0})}$. Therefore, $[\theta]\in{i_{\#}\pi_{1}(V,x_{0})}$ concludes that $\theta$ is a small loop and accordingly, $\pi_{1}^{s}(X,x_{0})=\pi_{1}^{sp}(X,x_{0})$.
\end{proof}

\begin{corollary}
If $X$ is strong SLT at $x_{0}$, $\pi_{1}^{s}(X,x_{0})=\pi_{1}^{sp}(X,x_{0})$.
\end{corollary}
\begin{proof}
It follows immediately from Proposition \ref{Pro1}.
\end{proof}

In the following example, we give a locally quasinormal subgroup which is not normal.

\begin{example}\label{EX}
As we know, Spanier subgroup $\pi(\mathcal{U},x_{0})$ and path Spanier subgroup $\widetilde{\pi}(\mathcal{V},x_{0})$ are equal if $\widetilde{\pi}(\mathcal{V},x_{0})$ is normal \cite{TorS}. Therefore, path Spanier subgroup is not necessarily normal. On the other hand, the form of elements of $\widetilde{\pi}(\mathcal{V},x_{0})$ and $\pi(\alpha, V_{\alpha})$ for each path $\alpha\in{P(X,x_{0})}$ and for every open subset $V_{\alpha}\in{\mathcal{V}}$ of $\alpha(1)$ implies that $\pi(\alpha,V_{\alpha})\widetilde{\pi}(\mathcal{V},x_{0})=\widetilde{\pi}(\mathcal{V},x_{0})\pi(\alpha,V_{\alpha})$. Accordingly, $\widetilde{\pi}(\mathcal{V},x_{0})$ is locally quasinormal.
\end{example}

%%%%%%%%%%%%%%%%%%%%%%%%%%%%%%%%%%%%%%%%%%%%%%%%%%

%\section*{Acknowledgments}

\section*{Reference}

\bibliography{mybibfile}



% \end{thebibliography}

\end{document}